\documentclass[12pt,reqno]{amsart}
\usepackage{amsmath, amsthm, amssymb}

\advance \topmargin by -\headheight
\advance \topmargin by -\headsep

\evensidemargin \oddsidemargin
\newtheorem{theorem}{Theorem}[section]
\newtheorem{lemma}[theorem]{Lemma}

\theoremstyle{definition}

\theoremstyle{remark}

\numberwithin{equation}{section}

\newcommand{\mmod}[1]{\,\,(\text{mod}\,\,#1)}

 \def\bfr{{\mathbf r}}

\def\calZ{{\mathcal Z}}

\def\dbN{{\mathbb N}}
\def\dbR{{\mathbb R}}
\def\dbZ{{\mathbb Z}}

\def\grB{{\mathfrak B}}
\def\grC{{\mathfrak C}}
\def\grD{{\mathfrak D}}

 \def\grI{{\mathfrak I}}

\def\grm{{\mathfrak m}}\def\grM{{\mathfrak M}}
\def\grS{{\mathfrak S}}
\def\grR{{\mathfrak R}}
\def\grB{{\mathfrak B}}\def\grC{{\mathfrak C}}
\def\grK{{\mathfrak K}}

\def\alp{{\alpha}} 
\def\bet{{\beta}}  
\def\gam{{\gamma}} \def\Gam{{\Gamma}}
\def\del{{\delta}} \def\Del{{\Delta}}

\def\tet{{\theta}}  
\def\kap{{\kappa}}
\def\lam{{\lambda}}

\def\sig{{\sigma}}  

\def\Ups{{\Upsilon}} 

\def\ome{{\omega}} \def\Ome{{\Omega}} 
\def\d{{\partial}}
\def\eps{\varepsilon}

\def\le{\leqslant} \def\ge{\geqslant}

\def\d{{\,{\rm d}}}

\begin{document}
\title[Waring's problem]{The asymptotic formula in Waring's problem: higher order expansions}
\author[R. C. Vaughan]{R. C. Vaughan$^*$}
\address{Department of Mathematics, McAllister Building, Pennsylvania State University, University Park,
PA 16802-6401, U.S.A.}
\email{rvaughan@math.psu.edu}
\author[T. D. Wooley]{T. D. Wooley}
\address{School of Mathematics, University of Bristol, University Walk, Clifton, Bristol BS8 1TW,
United Kingdom}
\email{matdw@bristol.ac.uk}
\thanks{$^*$Supported in part by NSA grant no. H98230-12-1-0276.}
\date{}
\begin{abstract} When $k>1$ and $s$ is sufficiently large in terms of $k$, we derive an explicit multi-term
 asymptotic expansion for the number of representations of a large natural number as the sum of $s$
 positive integral $k$th powers.
\end{abstract}
\maketitle

\section{Introduction} As is usual in Waring's problem, when $k>1$, we let $R_s(n)$ denote the number
 of representations of $n$ as the sum of $s$ $k$th powers of positive integers. Then, as first discovered by
 Hardy and Littlewood \cite{HL1922}, provided that $s$ is sufficiently large in terms of $k$, one has the
 asymptotic formula
\begin{equation}\label{1.1}
R_s(n)\sim {\frac{\Gamma(1+1/k)^s}{\Gamma(s/k)}}\grS_s(n)n^{s/k-1}
\end{equation}
as $n\rightarrow\infty$, where the {\it singular series} $\grS_s(n)$ is defined by
\begin{equation}\label{1.2}
\grS_s(n)=\sum_{q=1}^\infty \sum^q_{\substack{a=1\\(a,q)=1}}\Bigl( q^{-1}\sum_{r=1}^q
 e(ar^k/q)\Bigr)^se(-na/q).
\end{equation}
Here we use the familiar notation $e(z)=e^{2\pi iz}$. This asymptotic formula has been established by the
 first author \cite{Vau1986a, Vau1986b} for $s\ge 2^k$ $(k\ge 3)$, and by the second
\cite{Woo2012, Woo2013} for $s\ge 2k^2-2k-8$ $(k\ge 6)$. Meanwhile, Loh \cite{Loh1996} has
 demonstrated limitations to the quality of the error term which can be obtained in the formula (\ref{1.1}).
 In this memoir we explain the enigmatic phenomenon discovered by Loh by showing, for the first time,
that there are second and higher order terms present in the asymptotic expansion of $R_s(n)$. These new
 terms resemble the main term, though for odd $k$ there are intriguing differences.\par

Suppose that $n$ is a natural number sufficiently large in terms of $s$ and $k$, and define $P=n^{1/k}$.
 Let $\grM$ denote the union of the {\it major arcs}
$$\grM(q,a)=\{ \alp \in [0,1):|q\alp -a|\le P/(2kn)\},$$
with $0\le a\le q\le P$ and $(a,q)=1$, and define the {\it minor arcs} $\grm$ by means of the relation
$\grm=[0,1)\setminus \grM$. In addition, we introduce the Weyl sum
\begin{equation}\label{1.3}
f(\alp)=\sum_{1\le x\le P} e(\alp x^k).
\end{equation}
Finally, when $\nu$ is a real number, we say that the exponent $t$ is $\nu$-admissible for $k$ when
\begin{equation}\label{1.4}
\int_\grm |f(\alp)|^t\d\alp=o(P^{t-k-\nu})
\end{equation}
as $P\rightarrow\infty$. We note that, given a $0$-admissible exponent $s_0$, the asymptotic formula 
(\ref{1.1}) holds whenever $s\ge \max\{ s_0, 5, k+1\}$ (see \cite[Theorem 4.4]{Vau1997}).\par

When $\nu\ge 1$, define the exponent $\sigma_\nu(k)$ by
$$\sigma_\nu (k)=\begin{cases}
2^k+2^{k-1}\nu,&\text{when $2\le k\le 5$,}\\
2k^2-2+2^{k-1}(\nu-1),&\text{when $k=6,7$,}\\
4k-2+2k(k-2)\nu,&\text{when $k\ge 8$.}
\end{cases}$$
Then one may show that the exponent $s$ is $\nu$-admissible for $k$ when $s>\sigma_\nu (k)$. When
 $2\le k\le 5$, such follows from the classical approach of \cite[Chapters 2, 4]{Vau1997}. Indeed, a
 careful analysis of the methods underlying \cite{Vau1986a, Vau1986b} (incorporating refinements in
 \cite{Bok1993, HT1988, Hoo1996}) reveals that when $k\ge 3$ the exponent $\frac{3}{2}2^k$ is
$1$-admissible for $k$, and likewise that $2^{k+1}$ is $2$-admissible for $k$. When $k\ge 8$, on the
other hand, the above assertion follows by combining \cite[Theorems 10.1 and 11.1]{Woo2013}, whilst
 for $k=6,7$ one instead combines \cite[Theorem 10.1]{Woo2013} with Weyl's inequality (see
 \cite[Lemma 2.4]{Vau1997}).\par

In \S\S2--11 we enhance the familiar analysis of the major arc contribution in Waring's problem so as to
 derive higher order asymptotic expansions of shape
\begin{equation}\label{1.5}
R_s(n)=n^{s/k-1}(\grC_0+\grC_1n^{-1/k}+\ldots +\grC_Jn^{-J/k})+o(n^{(s-J)/k-1}),
\end{equation}
as $n\rightarrow\infty$. We divide our results according to whether $k$ is even or odd. Here and in what follows, we put $\del_k=1$ when $k=2$, and $\del_k=0$ when $k\ge 3$.

\begin{theorem}\label{theorem1.1} Let $k$ be even and $J\ge 0$. Suppose that $s$ is $J$-admissible for
 $k$ and $s\ge (J+1)(k+2)+\del_k$. Then one has the asymptotic formula (\ref{1.5}) with
\begin{equation}\label{1.6}
\grC_j=\left( -\tfrac{1}{2}\right)^j\binom{s}{j}\frac{\Gam(1+1/k)^{s-j}}{\Gam ((s-j)/k)}
\grS_{s-j}(n)\quad (0\le j\le J).
\end{equation}
\end{theorem}

Note that the singular series $\grS_{s-j}(n)$ in this statement is defined via (\ref{1.2}). We recall that
 when $s\ge 4$, the singular series $\grS_s(n)$ converges absolutely and is non-negative. Further, one has
 $\grS_s(n)\ll 1$ when $s\ge k+2+\del_k$, and $\grS_s(n)\ll n^\eps$ when $s=k+1+\del_k$. It is known
 that when $s\ge 4k$, the singular series satisfies the lower bound $\grS_s(n)\gg 1$, and that such
 remains true for $s\ge 5$ when $k=2$, for $s\ge 4$ when $k=3$, and for $s\ge \frac{3}{2}k$ when $k$
 is not a power of $2$. This lower bound also holds for the integer $n$ provided that $s\ge k+1+\del_k$
 and in addition, for every natural number $q$, the congruence
 $x_1^k+\dots +x_s^k\equiv n\mmod{q}$ possesses a solution with $(x_1,q)=1$ (see the final sections
 of \cite[Chapters 2 and 4]{Vau1997} for an account of such matters).\par

In order to describe our asymptotic formula when $k$ is odd, we must introduce a modified singular
 series. When $a\in \dbZ$ and $q\in \dbN$, define
\begin{equation}\label{1.7}
S(q,a)=\sum_{r=1}^qe(ar^k/q)\quad \text{and}\quad T(q,a)=\sum_{r=1}^q
\left( \frac{1}{2}-\frac{r}{q}\right) e(ar^k/q).
\end{equation}
Then, when $0\le j\le s$, we define the modified singular series
\begin{equation}\label{1.8}
\grS_{s,j}(n)=\sum_{q=1}^\infty \sum^q_{\substack{a=1\\ (a,q)=1}}(q^{-1}S(q,a))^{s-j}
T(q,a)^je(-na/q).
\end{equation}
Notice that $\grS_{s,0}(n)=\grS_s(n)$. We demonstrate in Lemma \ref{lemma10.2} that the singular series $\grS_{s,j}(n)$ is absolutely convergent
for $s\ge \tfrac{1}{2}(j+2)(k+2)$.

\begin{theorem}\label{theorem1.2} Let $k$ be odd, $k\ge 3$ and $0\le J\le k$. Suppose that the exponent $s$ is
$J$-admissible for $k$ and $s\ge (J+1)(k+2)$. Then one has the asymptotic formula (\ref{1.5}) with
\begin{equation}\label{1.9}
\grC_j=\binom{s}{j}\frac{\Gam(1+1/k)^{s-j}}{\Gam((s-j)/k)}\grS_{s,j}(n)\quad (0\le j\le J).
\end{equation}
\end{theorem}

Aficionados of the circle method will anticipate that similar conclusions may be obtained for almost all
 integers $n$ under weaker conditions on $s$.

\begin{theorem}\label{theorem1.3} Let $k\ge 2$ and $J\ge 0$. Suppose that $2s$ is $2J$-admissible for
 $k$ and $s\ge (J+1)(k+2)+\del_k$. Then one has the following conclusions.
\item{(i)} When $k$ is even, the asymptotic formula (\ref{1.5}) holds for almost all integers $n$ with
 $1\le n\le N$, with coefficients given by (\ref{1.6}).
\item{(ii)} When $k$ is odd and $J\le k$, the asymptotic formula (\ref{1.5}) holds for almost all integers
 $n$ with $1\le n\le N$, with coefficients given by (\ref{1.9}).
\end{theorem}

As we have noted, the main term in the asymptotic formula (\ref{1.5}) is classical. This much was
 established by Hardy and Littlewood \cite{HL1922} in their series of seminal papers concerning the
 application of their circle method to Waring's problem. Beyond this main term little was known until the
 work of Loh \cite{Loh1996}. This shows that when $k\ge 3$, one has
$$R_{k+1}(n)-\Gam(1+1/k)^k\grS_{k+1}(n)n^{1/k}=\Ome(n^{1/(2k)}),$$
and further that for $s\ge k+2$, one has
\begin{equation}\label{1.10}
R_s(n)-\frac{\Gam(1+1/k)^s}{\Gam(s/k)}\grS_s(n)n^{s/k-1}=\Ome_-(n^{(s-1)/k-1}).
\end{equation}
Theorem \ref{theorem1.1} shows that for even $k$, the $\Ome_-$-result (\ref{1.10}) is explained
precisely by the presence in the asymptotic formula (\ref{1.5}) of the secondary term
$$\grC_1n^{(s-1)/k-1}=-\tfrac{1}{2}s\frac{\Gam (1+1/k)^{s-1}}{\Gam ((s-1)/k)}
\grS_{s-1}(n)n^{(s-1)/k-1}.$$
When $k$ is odd, Theorem \ref{theorem1.2} shows instead that one has the secondary term
$$\grC_1n^{(s-1)/k-1}=s\frac{\Gam(1+1/k)^{s-1}}{\Gam((s-1)/k)}\grS_{s,1}(n)n^{(s-1)/k-1}.$$
Presumably, the modified singular series $\grS_{s,1}(n)$ is non-zero under modest conditions, and this
 would again precisely explain Loh's discovery. However, when $k$ is odd this series does not have an
 interpretation as an Euler product, and so in general it is not entirely clear how it behaves. In \S13 we
 explore what can be said concerning the behaviour of $\grS_{s,j}(n)$.

\begin{theorem}\label{theorem1.4}
Suppose that $k$ is odd and $s\ge \tfrac{3}{2}k+3$. Let $Q$ be a positive integer, and let $n$ be a
 multiple of $Q!$. Then one has
$$\grS_{s,1}(n)=-\tfrac{1}{2}\grS_{s-1}(n)+O(Q^{-1/(2k)}).$$
\end{theorem}

When $k$ is odd and $s\ge \tfrac{3}{2}k+3$, the singular series $\grS_{s-1}(n)$ is positive and bounded
 away from zero (see \cite[Theorem 4.6]{Vau1997}). By taking $Q=Q(s,k)$ to be sufficiently large in
 terms of $s$ and $k$, it therefore follows that $-\grS_{s,1}(n)\gg1$ for a positive proportion of $n$. If 
instead one takes $Q$ to grow slowly with $n$, say $Q=\sqrt{\log \log n}$, one has $Q!=o(\log n)$, and hence there are at least $N/\log N$ integers with $1\le n\le N$ for which
$\grS_{s,1}(n)=-\tfrac{1}{2}\grS_{s-1}(n)+O(1/\log \log \log N)$. Thus $\grS_{s,1}(n)$ is frequently very close to $-\tfrac{1}{2}\grS_{s-1}(n)$. One is tempted to believe that in fact this is usually
 the case, but in any case Loh's conclusion (\ref{1.10}) is explained by this observation for odd $k$.\par

Finally, we show that the modified singular series $\grS_{s,j}(n)$ is often non-zero for small values of $j$.

\begin{theorem}\label{theorem1.5} Suppose that $j\ge 0$ and $s\ge  \frac{1}{2}(j+4)(k+2)$. Then
 there is a constant $C_j>0$ such that, for all sufficiently large $x$, the number $N_j(x)$ of integers $n$
 with $1\le n\le x$ for which $|\grS_{s,j}(n)|\ge C_j$ satisfies $N_j(x)\ge C_jx$.
\end{theorem}

The reader having a passing familiarity with the theory of modular forms will recognise that in the case
 $k=2$, corresponding to the representation of integers as sums of squares, very precise asymptotic
 formulae are available involving the Fourier coefficients of Eisenstein series and cusp forms (see, for
 example, \cite[\S11.3]{Iwa1997}). This observation might prompt speculation that some exotic
 generalisation of Eisenstein series and cusp forms might conceivably describe $R_s(n)$ also when
$k\ge 3$. When $k$ is even, the asymptotic formula (\ref{1.5}) supplied by Theorem \ref{theorem1.1}
 seems consistent with this speculation, since each term is given by a classical singular series having an
 Euler product interpretation. When $k$ is odd, however, the modified singular series $\grS_{s,j}(n)$ pose
 interesting problems for such an explanation. Perhaps the exponential sums
$$\sum_{r=1}^q\psi(r/q)e(ar^k/q),$$
in which $\psi\in \dbZ[x]$ has positive degree, demand further investigation.\par

This paper is organised as follows. In \S\S2--6 we examine even $k$. Following some preliminary
 discussion in \S2, we establish basic major arc estimates in \S3. Certain auxiliary estimates require
multi-term asymptotic expansions, and so in \S4 we apply Euler-MacLaurin expansions, inserting the
output into corresponding major arc estimates in \S5. We complete the proof of Theorem
\ref{theorem1.1} in \S6 by combining the contributions of these estimates. The treatment of odd $k$ in
 \S\S7--11 mirrors that of even $k$, though in \S10 we briefly discuss the novel modified singular series
 $\grS_{s,j}(n)$. This establishes Theorem \ref{theorem1.2}. In \S12 we discuss exceptional sets, proving
 Theorem \ref{theorem1.3}. Finally, in \S13, we investigate the singular series $\grS_{s,j}(n)$ for odd 
$k$, establishing Theorems \ref{theorem1.4} and \ref{theorem1.5}.\par

Our basic parameter is $P$, a sufficiently large positive number, and we will normally take $P=n^{1/k}$. 
 Exceptionally in \S12 we will take $P=N^{1/k}$.  In the $o$-notation the limiting process will invariably
 be as $P\rightarrow\infty$, or equivalently $n$ or $N\rightarrow\infty$.  In this paper, implicit constants
 in Vinogradov's notation $\ll$ and $\gg$ may depend on $s$, $k$ and $\eps$. Whenever $\eps$ appears
 in a statement, either implicitly or explicitly, we assert that the statement holds for each $\eps>0$. Finally,
 we write $\|\tet\|=\underset{m\in \dbZ}\min |\tet-m|$.

\section{Preliminary man\oe uvres, for even $k$} Suppose that $s$ and $k$ are natural numbers with
 $s>k\ge 2$ and $k$ even. We establish the multi-term asymptotic formula claimed in Theorem
 \ref{theorem1.1} by applying the Hardy-Littlewood method to analyse a modification of the standard
 Waring problem. Let $n$ be a positive integer sufficiently large in terms of $s$ and $k$. We recall that
 $P=n^{1/k}$, and define $R_s^*(n)$ to be the number of integral representations of $n$ in the shape
\begin{equation}\label{2.1}
n=x_1^k+\ldots +x_s^k,
\end{equation}
with $|x_i|\le P$ $(1\le i\le s)$. It is apparent that $R_s^*(n)$ is approximately $2^sR_s(n)$. On
 accounting for the contribution arising from those representations in which one or more variables are
 zero, we find that
\begin{equation}\label{2.2}
R_s^*(n)=\sum_{r=0}^s2^{s-r}\binom{s}{r}R_{s-r}(n),
\end{equation}
whence
\begin{equation}\label{2.3}
R_s(n)=2^{-s}\sum_{r=0}^s(-1)^r\binom{s}{r}R_{s-r}^*(n).
\end{equation}
Indeed, on substituting (\ref{2.2}) into the right hand side of (\ref{2.3}), we see that
\begin{align*}
2^{-s}\sum_{r=0}^s(-1)^r\binom{s}{r}R_{s-r}^*(n)&=2^{-s}\sum_{r=0}^s(-1)^r
\binom{s}{r}\sum_{l=0}^{s-r}2^{s-r-l}\binom{s-r}{l}R_{s-r-l}(n)\\
&=\sum_{u=0}^s2^{-u}\binom{s}{u}R_{s-u}(n)\sum_{l=0}^u(-1)^{u-l}\binom{u}{l}.
\end{align*}
The innermost sum on the right hand side is equal to $(1-1)^u$, and so the only non-zero term in the
 outermost sum is that with $u=0$. The claimed relation (\ref{2.3}) therefore follows. In order to
 establish Theorem \ref{theorem1.1}, it suffices to obtain a multi-term asymptotic expansion for
 $R_t^*(n)$ when $t$ is close to $s$. This, it transpires, is more easily accomplished than the analogous
 task for $R_t(n)$.\par

Next, define the generating function
\begin{equation}\label{2.4}
h(\alp)=\sum_{|x|\le P}e(\alp x^k),
\end{equation}
and, when $\grB\subseteq [0,1)$ is measurable, put
\begin{equation}\label{2.5}
R_t^*(n;\grB)=\int_\grB h(\alp)^te(-n\alp)\d\alp .
\end{equation}
By orthogonality, we have
\begin{equation}\label{2.6}
R_t^*(n)=R_t^*(n;\grM)+R_t^*(n;\grm).
\end{equation}
The hypotheses of the statement of Theorem \ref{theorem1.1} permit us the assumption that $s$ is
$J$-admissible for $k$, and this, in essence, takes care of the analysis of $R_t^*(n;\grm)$. The lemma
 below formalises this observation. We note for future reference that, in view of (\ref{1.3}) and (\ref{2.4}),
 one has the relation
\begin{equation}\label{2.7}
h(\alp)=1+2f(\alp).
\end{equation}

\begin{lemma}\label{lemma2.1} Suppose that $J\ge 0$ and $s\ge k+J+1$.  Then for $0\le t\le s-k-J-1$, 
one has
\begin{equation}\label{2.8}
R_t^*(n;[0,1))\ll P^{s-k-J-1}.
\end{equation}
Suppose instead that $s$ is $J$-admissible for $k$ and $s-k-J\le t\le s$.  Then
\begin{equation}\label{2.9}
R_t^*(n;\grm)=o(P^{s-k-J}).
\end{equation}
\end{lemma}

\begin{proof} An application of the triangle inequality within (\ref{2.5}) leads, via (\ref{2.7}), to the 
bound
$$R_t^*(n;\grB)\ll 1+\int_\grB |f(\alp)|^t\d\alp .$$
When $0\le t\le s-k-J-1$, therefore, the trivial bound $|f(\alp)|\le P$ yields (\ref{2.8}). When instead 
$s-k-J\le t\le s$, one finds from H\"older's inequality that
$$R_t^*(n;\grm)\ll 1+ \Bigl( \int_\grm |f(\alp)|^s\d\alp \Bigr)^{t/s}
\Bigl( \int_0^1\d\alp \Bigr)^{1-t/s}.$$
In the first integral on the right hand side, we invoke the hypothesis that $s$ is $J$-admissible for $k$, 
and apply the associated estimate (\ref{1.4}) with $t$ replaced by $s$. Since $t/s\le 1$  we obtain 
(\ref{2.9}). This completes the proof of the lemma.
\end{proof}

By substituting the conclusions of Lemma \ref{lemma2.1} into (\ref{2.3}), and noting (\ref{2.6}), we 
deduce that when $s$
is $J$-admissible and $s\ge k+J+1$, then
\begin{equation}\label{2.10}
R_s(n)=2^{-s}\sum_{r=0}^{k+J}(-1)^r\binom{s}{r}R_{s-r}^*(n;\grM)+o(P^{s-k-J}).
\end{equation}
Thus it remains to analyse $R_{s-r}^*(n;\grM)$ for $0\le r\le k+J$.

\section{The major arc contribution truncated, for even $k$} Our first step in the analysis of 
$R_{s-r}^*(n;\grM)$ is the replacement of the generating function $h(\alp)$ in (\ref{2.5}) by a suitable
 approximation. This requires a little preparation. Define $S(q,a)$ as in (\ref{1.7}), and when 
$\bet\in \dbR$ put
\begin{equation}\label{3.1}
v(\bet)=\int_0^Pe(\bet \gam^k)\d\gam .
\end{equation}
We define $f^*(\alp)$ for $\alp\in \grM$ by taking
\begin{equation}\label{3.2}
f^*(\alp)=q^{-1}S(q,a)v(\alp-a/q)\quad \text{when}\quad \alp\in \grM(q,a).
\end{equation}
From \cite[Theorem 4.1]{Vau1997}, it therefore follows that when $0\le a\le q\le P$ and $(a,q)=1$, one
 has
\begin{equation}\label{3.3}
\sup_{\alp\in \grM(q,a)}|f(\alp)-f^*(\alp)|\ll q^{1/2+\eps}\le P^{1/2+\eps},
\end{equation}
whence (\ref{2.7}) yields
\begin{equation}\label{3.4}
\sup_{\alp\in \grM}|h(\alp)-2f^*(\alp)|\ll P^{1/2+\eps}.
\end{equation}

\par An application of the binomial theorem within (\ref{2.5}) reveals that for non-negative integers $t$,
 one has
\begin{equation}\label{3.5}
R_t^*(n;\grM)=\sum_{l=0}^t\binom{t}{l}\grI_{t,l}(n),
\end{equation}
where
\begin{equation}\label{3.6}
\grI_{t,l}(n)=\int_\grM (2f^*(\alp))^{t-l}(h(\alp)-2f^*(\alp))^le(-n\alp)\d\alp .
\end{equation}

\begin{lemma}\label{lemma3.1} Suppose that $k\ge 2$, and that $J$ and $r$ are non-negative integers.
 Then whenever $l>2J-2r$ and $s\ge \max \{l+r,k+2J+4\}$, one has
$$\grI_{s-r,l}(n)=o(P^{s-k-J}).$$
\end{lemma}

\begin{proof} When $k=2$, the methods of \cite[Chapter 4]{Vau1997} deliver the upper bound
\begin{equation}\label{3.7}
\int_\grM |f^*(\alp)|^{k+2}\d\alp \ll P^{2+\eps},
\end{equation}
a bound that may be confirmed also when $k\ge 3$ by the methods underlying the proof of 
\cite[Lemma 5.1]{Vau1989}. We apply this estimate in order to simplify the estimation of the integral
 $\grI_{s-r,l}(n)$.\par

We begin by considering the situation in which $s\ge k+r+l+2$. Here, by applying the trivial bound
 $f^*(\alp)\ll P$ in (\ref{3.6}), and then utilising (\ref{3.4}) and (\ref{3.7}), we obtain the estimate
\begin{align*}
\grI_{s-r,l}(n)&\ll \Bigl( \sup_{\alp\in \grM}|h(\alp)-2f^*(\alp)|\Bigr)^l
\Bigl( P^{s-k-r-l-2}\int_\grM |f^*(\alp)|^{k+2}\d\alp \Bigr)\\
&\ll P^{s-k-r-l/2+\eps}.
\end{align*}
The hypothesis $l>2J-2r$ therefore ensures that
$$\grI_{s-r,l}(n)\ll P^{s-k-J-1/2+\eps}=o(P^{s-k-J}).$$
This completes the proof of the lemma in this first situation.\par

It remains to consider those circumstances in which
$$r+l\le s\le k+r+l+1.$$
Here we may assume without loss that $l\ge 2J-2r+3$, for if instead one were to have $l\le 2J-2r+2$, 
then
$$s\le k+r+(2J-2r+2)+1\le k+2J+3,$$
contradicting the hypothesis $s\ge k+2J+4$. Note next that the measure of $\grM$ is $O(P^{2-k})$. Let
 $\ome=(s-r-l)/(k+2)$. Then, by applying H\"older's inequality to (\ref{3.6}), we obtain
$$\grI_{s-r,l}(n)\ll \Bigl( \sup_{\alp \in \grM}|h(\alp)-2f^*(\alp)|\Bigr)^l
\Bigl( \int_\grM |f^*(\alp)|^{k+2}\d\alp \Bigr)^\ome \Bigl(\int_\grM\d\alp \Bigr)^{1-\ome}.$$
We therefore find from (\ref{3.4}) and (\ref{3.7}) that $\grI_{s-r,l}(n)=O(P^{\lam+\eps})$, where
\begin{equation}\label{3.8}
\lam=s-k-r-\tfrac{1}{2}l+2-2(s-r-l)/(k+2).
\end{equation}
We now divide into cases.\par

First, when $l>2J-2r+4$, it follows from the hypothesis $s\ge l+r$ that
$$l>2J-2r+4-4(s-r-l)/(k+2),$$
whence
$$r+\tfrac{1}{2}l-2+2(s-r-l)/(k+2)>J.$$
Thus we deduce from (\ref{3.8}) that $\lam<s-k-J$, so that $\grI_{s-r,l}(n)=o(P^{s-k-J})$.\par

Otherwise, we have $2J-2r+3\le l\le 2J-2r+4$. Here, if one were to have $s\le \frac{1}{4}(k+2)+r+l$,
then we obtain
$$s\le k+2J-r+\tfrac{1}{4}(18-3k)\le k+2J+3,$$
contradicting the hypothesis $s\ge k+2J+4$. Then we have $s>\frac{1}{4}(k+2)+r+l$, so from 
(\ref{3.8}) we infer that
$$\lam<s-k-r-\tfrac{1}{2}(l-3)\le s-k-r-(J-r).$$
Thus, in this final situation, we again deduce that $\grI_{s-r,l}(n)=o(P^{s-k-J})$, and the proof of the
 lemma is complete.
\end{proof}

Notice that when $s\ge k+2J+4$ and $r>J$, the conclusion of Lemma \ref{lemma3.1} ensures that
$\grI_{s-r,l}(n)=o(P^{s-k-J})$. Thus, on combining (\ref{2.10}) and (\ref{3.5}) with Lemma 
\ref{lemma3.1}, we deduce that whenever $s\ge k+2J+4$ one has
\begin{equation}\label{3.9}
R_s(n)=2^{-s}\sum_{r=0}^J(-1)^r\binom{s}{r}\sum_{l=0}^{2J-2r}\binom{s-r}{l}
\grI_{s-r,l}(n)+o(P^{s-k-J}).
\end{equation}

\section{An auxiliary lemma, for even $k$}
Before proceeding further, we must estimate certain multiple sums over arithmetic progressions. We first
 recall two standard tools, beginning with the Euler-MacLaurin formula. The associated Bernoulli numbers
 $B_\kap$ $(\kap\ge 0)$ may be defined by putting $B_0=1$, $B_1=-\frac{1}{2}$, and iterating the 
relation
$$B_\kap=\sum_{j=0}^\kap \binom{\kap}{j}B_{\kap-j}\quad (\kap\ge 2).$$
The Bernoulli polynomials $B_\kap(x)$ may then be defined by taking
$$B_\kap(x)=\sum_{j=0}^\kap \binom{\kap}{j}B_{\kap-j}x^j\quad (\kap\ge 0).$$
We write $\{x\}=x-\lfloor x\rfloor$, where $\lfloor x\rfloor $ denotes the greatest integer no larger than
 $x$, and write $\lceil x\rceil$ for the least integer no smaller than $x$. It is convenient then to write
 $\bet_\kap(x)=B_\kap(\{ x\})$ for $\kap\ge 0$.

\begin{lemma}\label{lemma4.1} Let $a$ and $b$ be real numbers with $a<b$, and let $K$ be a positive
 integer. Suppose that $F$ has continuous derivatives through the $(K-1)$-st order on $[a,b]$, that the
 $K$-th derivative of $F$ exists and is continuous on $(a,b)$, and $|F^{(K)}(x)|$ is integrable on $[a,b]$.
 Then
\begin{align*}
\sum_{a<n\le b}F(n)=&\, \int_a^bF(x)\d x+\sum_{\kap=1}^K\frac{(-1)^\kap}{\kap !}
\left( \bet_\kap (b)F^{(\kap-1)}(b)-\bet_\kap(a)F^{(\kap-1)}(a)\right) \\
&\, -\frac{(-1)^K}{K!}\int_a^b\bet_K(x)F^{(K)}(x)\d x.
\end{align*}
\end{lemma}

\begin{proof} This is essentially the version of the Euler-MacLaurin summation formula provided in
\cite[Theorem B.5]{MV2007}. The statement of the latter demands that $F^{(K)}(x)$ exist and be
 continuous on $[a,b]$.
However, the argument of the proof of \cite[Theorem B.5]{MV2007} remains applicable if instead
 $F^{(K)}(x)$ exists and is continuous on $(a,b)$, and in addition $|F^{(K)}(x)|$ is integrable on $[a,b]$.
\end{proof}

Next, we recall Fa\`a di Bruno's formula for the $N$-th derivative of a composition of functions.

\begin{lemma}\label{lemma4.2} Suppose that $F$ and $G$ have continuous derivatives of order up to the
 $N$th on an open interval containing $x$. Then
$$\frac{{\rm d}^N\ }{{\rm d}x^N}F(G(x))=\sum \frac{N!}{m_1!\ldots m_N!}
F^{(m_1+\ldots +m_N)}(G(x))\prod_{j=1}^N\left( \frac{G^{(j)}(x)}{j!}\right)^{m_j},$$
where the summation is over non-negative integers $m_1,\ldots ,m_N$ satisfying
$$m_1+2m_2+\ldots +Nm_N=N.$$
\end{lemma}

\begin{proof} See \cite{Joh2002} for an account of this formula and its history.
\end{proof}

We apply Lemmata \ref{lemma4.1} and \ref{lemma4.2} in combination to obtain an asymptotic formula
 for an important auxiliary sum. Let $X$ be a positive real number, and let $\tet$ be a non-negative real
 exponent. When $q\in \dbN$ and $r\in \dbZ$, we define
\begin{equation}\label{4.1}
\Ups_{q,r}(X;\tet)=\sum_{-(X+r)/q\le h\le (X-r)/q}(X^k-(qh+r)^k)^\tet.
\end{equation}

\begin{lemma}\label{lemma4.3}
When $1\le N\le \lceil \tet\rceil $, one has
$$\Ups_{q,r}(X;\tet)=(2X/q)X^{k\tet}\frac{\Gam(1+\tet)\Gam(1+1/k)}{\Gam(1+\tet+1/k)}+
O\left( X^{k\tet}(q/X)^{N-1}\right).$$
\end{lemma}

\begin{proof} We apply Lemma \ref{lemma4.2} with
\begin{equation}\label{4.1a}
F(y)=y^\tet \quad \text{and}\quad G(x)=X^k-(qx+r)^k.
\end{equation}
Write $a=-(X+r)/q$ and $b=(X-r)/q$. Then one finds that
$$G(a)=X^k-(-X)^k=0\quad \text{and}\quad G(b)=X^k-X^k=0.$$
Further, when $0\le j\le k$, one has
\begin{equation}\label{4.2}
G^{(j)}(x)=-\frac{k!}{(k-j)!}q^j(qx+r)^{k-j},
\end{equation}
whilst $G^{(j)}(x)=0$ for $j>k$. Also, when $0\le m\le \lceil \tet\rceil$, one has
\begin{equation}\label{4.3}
F^{(m)}(y)=\tet(\tet-1)\ldots (\tet-m+1)y^{\tet-m},
\end{equation}
where the condition $y\ne 0$ should be imposed in case $m>\tet$. It follows from Lemma 
\ref{lemma4.2} that $F(G(x))$ has continuous derivatives through the $N$-th order on $(a,b)$,
 continuous derivatives through the $(N-1)$-st order on $[a,b]$, and further 
$|{\rm d}^N F(G(x))/{\rm d}x^N|$ is integrable on $[a,b]$. Note also that when $0\le m<\tet$, one has
 $F^{(m)}(G(a))=F^{(m)}(G(b))=0$. Then Lemma \ref{lemma4.2} shows that
$$\left.\frac{{\rm d}^\kap\ }{{\rm d}x^\kap}F(G(x))\right|_{x=a}=
\left.\frac{{\rm d}^\kap\ }{{\rm d}x^\kap}F(G(x))\right|_{x=b}=0\quad  (0\le \kap <\tet).$$
On substituting these conclusions into Lemma \ref{lemma4.1}, we see that
\begin{equation}\label{4.4}
\sum_{a\le h\le b}F(G(h))=\int_a^bF(G(x))\d x -\frac{(-1)^N}{N!}\int_a^b \bet_N(x)\frac{{\rm d}^N\ }
{{\rm d}x^N}F(G(x))\d x.
\end{equation}

\par The first term on the right hand side of (\ref{4.4}) is easily evaluated. By making the change of
 variable $y=(qx+r)/X$, we find that
\begin{align}
\int_a^bF(G(x))\d x&=q^{-1}X^{k\tet+1}\int_{-1}^1(1-y^k)^\tet \d y\notag \\
&=(2X/q)X^{k\tet}\frac{\Gam(1+\tet)\Gam(1+1/k)}{\Gam(1+\tet+1/k)}.\label{4.5}
\end{align}

\par For the second term we must work harder. When $\tet$ is an integer, it follows from (\ref{4.2}) and
 (\ref{4.3}) via Lemma \ref{lemma4.2} that one has the upper bound
$$\sup_{a\le x\le b}\left| \frac{{\rm d}^N\ }{{\rm d}x^N}F(G(x))\right| \ll q^NX^{k\tet-N}.$$
When $\tet$ is not an integer, on the other hand, say $\{\tet\}=1-\nu$, then we find in like manner that
$$\sup_{(a+b)/2\le x\le b}\left| (X-qx-r)^\nu \frac{{\rm d}^N\ }{{\rm d}x^N}F(G(x))\right| 
\ll q^NX^{k\tet-N+\nu}$$
and
$$\sup_{a\le x\le (a+b)/2}\left| (X+qx+r)^\nu \frac{{\rm d}^N\ }{{\rm d}x^N}F(G(x))\right| 
\ll q^NX^{k\tet-N+\nu}.$$
Thus we deduce that
\begin{align}
\int_a^b&\bet_N(x)\frac{{\rm d}^N\ }{{\rm d}x^N}F(G(x))\d x\notag \\
&\ll q^NX^{k\tet -N+\nu}\Bigl( \int_a^b(X+qx+r)^{-\nu}+(X-qx-r)^{-\nu}\d x\Bigr) \notag\\
&\ll (q^NX^{k\tet-N+\nu})(q^{-1}X^{1-\nu})\ll X^{k\tet}(q/X)^{N-1}.\label{4.6}
\end{align}
Since this estimate is immediate when $\tet$ is an integer, the conclusion of the lemma follows on 
substituting (\ref{4.5}) and (\ref{4.6}) into (\ref{4.4}), and then recalling the definition (\ref{4.1}) of
 $\Ups_{q,r}(X;\tet)$.
\end{proof}

This lemma may be extended by induction to derive a multidimensional generalisation. When $q\in \dbN$ and
$r_1,\ldots ,r_l\in \dbZ$, we define
\begin{equation}\label{4.7}
\Xi_{q,\bfr}^{(l)}(X;\tet)=\sum_{\substack{|x_1|\le X\\ x_1\equiv r_1\mmod{q}}}\ldots
\sum_{\substack{|x_l|\le X\\ x_l\equiv r_l\mmod{q}}}(X^k-x_1^k-\ldots -x_l^k)^\tet ,
\end{equation}
where the summands are constrained by the inequality $x_1^k+\ldots +x_l^k\le X^k$.

\begin{lemma}\label{lemma4.4} When $1\le N\le \lceil \tet\rceil$, one has
\begin{align*}
\Xi_{q,\bfr}^{(l)}(X;\tet)=(2X/q)^lX^{k\tet}&\frac{\Gam(1+\tet)\Gam(1+1/k)^l}{\Gam(1+\tet+l/k)}\\
&\,+O\left( X^{k\tet}(q/X)^{N-1}(1+X/q)^{l-1}\right) .
\end{align*}
\end{lemma}

\begin{proof} We proceed by induction on $l$, noting that the case $l=1$ is already established by
 Lemma \ref{lemma4.3}. Suppose then that $L>1$, and that the desired conclusion has been established
 for $1\le l<L$. From (\ref{4.7}), we obtain
\begin{equation}\label{4.8}
\Xi_{q,\bfr}^{(L)}(X;\tet)=\sum_{-(X+r_L)/q\le h_L\le (X-r_L)/q}\Xi_{q,\bfr'}^{(L-1)}(Y;\tet),
\end{equation}
in which we have written
$$\bfr'=(r_1,\ldots ,r_{L-1})\quad \text{and}\quad Y=(X^k-(qh_L+r_L)^k)^{1/k}.$$
Our inductive hypothesis supplies the asymptotic formula
\begin{align*}
\Xi_{q,\bfr'}^{(L-1)}(Y;\tet)=(2Y/q)^{L-1}&Y^{k\tet}
\frac{\Gam(1+\tet)\Gam(1+1/k)^{L-1}}{\Gam(1+\tet+(L-1)/k)}\\
&+O\left( Y^{k\tet}(q/Y)^{N-1}(1+Y/q)^{L-2}\right) .
\end{align*}
By substituting this expression into (\ref{4.8}), we deduce that
\begin{equation}\label{4.9}
\Xi_{q,\bfr}^{(L)}(X;\tet)=\frac{\Gam(1+\tet)\Gam(1+1/k)^{L-1}}{\Gam(1+\tet+(L-1)/k)}T_0
+O(X^{k\tet}(q/X)^{N-1}(1+X/q)^{L-1}),
\end{equation}
where
\begin{equation}\label{4.10}
T_0=(2/q)^{L-1}\sum_{-(X+r_L)/q\le h_L\le (X-r_L)/q}\left(X^k-(qh_L+r_L)^k\right)^{\tet+(L-1)/k}.
\end{equation}

\par An application of Lemma \ref{lemma4.3} leads from (\ref{4.10}) to the asymptotic relation
$$T_0=(2X/q)^LX^{k\tet}\frac{\Gam(1+\tet+(L-1)/k)\Gam(1+1/k)}{\Gam(1+\tet+L/k)}
+O\left( X^{k\tet}(q/X)^{N-L}\right) .$$
We therefore infer from (\ref{4.9}) that the inductive hypothesis holds for $l=L$, confirming the inductive
 step and completing the proof of the lemma.
\end{proof}

\section{The major arc contribution evaluated, for even $k$} Our goal in this section is the evaluation of
 the integral $\grI_{t,l}(n)$ defined in (\ref{3.6}). With this objective in mind, we consider the auxiliary
 integral
\begin{equation}\label{5.1}
\grK_{u,l}(n)=\int_\grM (2f^*(\alp))^uh(\alp)^le(-n\alp)\d\alp .
\end{equation}
Note that, on recalling the definition (\ref{2.4}) of $h(\alp)$, this integral may be rewritten in the shape
\begin{equation}\label{5.2}
\grK_{u,l}(n)=2^u\sum_{|m_1|\le P}\ldots \sum_{|m_l|\le P}\grR_u(n-m_1^k-\ldots -m_l^k),
\end{equation}
where
\begin{equation}\label{5.3}
\grR_u(m)=\int_\grM f^*(\alp)^ue(-m\alp)\d\alp .
\end{equation}
Before refining the conventional major arc analysis of $\grR_u(m)$ so as to extract a sharper error term,
 we pause to record two estimates for the auxiliary sum
\begin{equation}\label{5.4}
V_A^B(u;\tet)=\sum_{A\le q<B}\sum^q_{\substack{a=1\\ (a,q)=1}}q^\tet|q^{-1}S(q,a)|^u.
\end{equation}

\begin{lemma}\label{lemma5.1} When $u>k+1+\del_k$ one has
$$V_1^Q(u;\tet)\ll 1+Q^{1+\tet-(u-1-\del_k)/k+\eps },$$
and when $u>k(1+\tet)+1+\del_k$ one has
$$V_Q^\infty(u;\tet)\ll Q^{1+\tet-(u-1-\del_k)/k+\eps }.$$
\end{lemma}

\begin{proof} The conclusion of \cite[Lemma 4.9]{Vau1997} supplies the bound
$$\sum_{1\le q\le 2Q}\sum^q_{\substack{a=1\\ (a,q)=1}}|q^{-1}S(q,a)|^{k+1+\del_k}\ll Q^\eps.$$
Then it follows from \cite[Theorem 4.2]{Vau1997} that
$$\sum_{Q\le q<2Q}q^\tet \sum^q_{\substack{a=1\\ (a,q)=1}}|q^{-1}S(q,a)|^u
\ll Q^{1+\tet-(u-1-\del_k)/k+\eps},$$
and the desired estimates follow by summing over dyadic intervals.
\end{proof}

Before announcing our refinement of the conventional major arc analysis, we define for future reference 
the truncated singular series
\begin{equation}\label{5.5}
\grS_u(m;P)=\sum_{1\le q\le P}\sum^q_{\substack{a=1\\ (a,q)=1}}\left(q^{-1}S(q,a)\right)^u
e(-ma/q).
\end{equation}

\begin{lemma}\label{lemma5.2}
Suppose that $u$ is an integer with $u\ge (J+1)k+2+\del_k$. Then
$$\grS_u(m;P)=\grS_u(m)+O(P^{-J-1/(2k)}).$$
Also, there is a positive number $\eta$ such that, whenever $|m|\le un$, one has
$$\grR_u(m)=\Del_m \frac{\Gam(1+1/k)^u}{\Gam(u/k)}\grS_u(m)m^{u/k-1}
+O(P^{u-k-J-\eta}),$$
where $\Del_m=1$ when $m\ge 0$, and $\Del_m=0$ when $m<0$.
\end{lemma}

\begin{proof} On recalling (\ref{3.2}) and (\ref{5.3}), we see that
\begin{equation}\label{5.6}
\grR_u(m)=\sum_{1\le q\le P}\sum^q_{\substack{a=1\\ (a,q)=1}}(q^{-1}S(q,a))^ue(-ma/q)I_u(m;q),
\end{equation}
where
$$I_u(m;q)=\int_{-P/(2kqn)}^{P/(2kqn)}v(\bet)^ue(-\bet m)\d\bet .$$
Define
\begin{equation}\label{5.7}
I(m)=\int_{-\infty}^\infty v(\bet)^ue(-\bet m)\d\bet .
\end{equation}
This integral is absolutely convergent for $u\ge k+1$, as is immediate from \cite[Theorem 7.3]{Vau1997}.
 The latter theorem also yields the estimate
$$I_u(m;q)-I(m)\ll P^u\int_{P/(2kqn)}^\infty (1+P^k\bet)^{-u/k}\d\bet \ll (qn/P)^{u/k-1}.$$
On substituting this relation into (\ref{5.6}) and then recalling (\ref{5.4}) and (\ref{5.5}), therefore, we
 obtain
\begin{align*}
\grR_u(m)-\grS_u(m;P)I(m)&\le \sum_{1\le q\le P}|I_u(m;q)-I(m)|
\sum^q_{\substack{a=1\\ (a,q)=1}}|q^{-1}S(q,a)|^u\\
&\ll (n/P)^{u/k-1}V_1^{P+1}(u;u/k-1).
\end{align*}
Our hypothesis concerning $u$ ensures that $u/k-1\ge J+(2+\del_k)/k$, and thus we discern from Lemma
 \ref{lemma5.1} that
\begin{equation}\label{5.8}
\grR_u(m)-\grS_u(m;P)I(m)\ll P^{u-k-J-1/(2k)}.
\end{equation}

\par The integral (\ref{5.7}) is the familiar singular integral in Waring's problem. In the integral form
 (\ref{3.1}) in which we have defined the generating function $v(\bet)$, a classical treatment of the type
 described on \cite[pages 21--23]{Dav2005} yields the formula
$$I(m)=\Del_m\frac{\Gam(1+1/k)^u}{\Gam(u/k)}m^{u/k-1}.$$
Also, our hypothesis on $u$ leads from (\ref{1.2}) and (\ref{5.5}) via (\ref{5.4}) and Lemma
 \ref{lemma5.1} to the bound
\begin{align*}
\grS_u(m)-\grS_u(m;P)&\le V_P^\infty(u;0)\ll P^{1-(u-1-\del_k)/k+\eps}\ll P^{-J-1/(2k)}.
\end{align*}
The proof of the lemma follows by substituting these estimates into (\ref{5.8}).
\end{proof}

This lemma may be combined with Lemma \ref{lemma4.4} in order to obtain an asymptotic formula for
 $\grK_{u,l}(n)$.

\begin{lemma}\label{lemma5.3}
Suppose that $u$ is an integer with $u\ge (J+1)k+2+\del_k$. Then there is a positive number $\eta$ for
 which
$$\grK_{u,l}(n)=2^{u+l}\frac{\Gam(1+1/k)^{u+l}}{\Gam((u+l)/k)}
\grS_{u+l}(n)n^{(u+l)/k-1}+O(P^{u+l-k-J-\eta}).$$
\end{lemma}

\begin{proof} On recalling the formula (\ref{5.2}) for $\grK_{u,l}(n)$, we find from Lemma
 \ref{lemma5.2} that there is a positive number $\eta$ such that
\begin{equation}\label{5.9}
\grK_{u,l}(n)=2^u\frac{\Gam(1+1/k)^u}{\Gam(u/k)}T_1+O(P^{u+l-k-J-\eta}),
\end{equation}
where
$$T_1=\underset{m_1^k+\ldots +m_l^k\le n}{\sum_{|m_1|\le P}\ldots \sum_{|m_l|\le P}}
\grS_u(n-m_1^k-\ldots -m_l^k)(n-m_1^k-\ldots -m_l^k)^{u/k-1}.$$
Applying the definition (\ref{1.2}) of the singular series, we find that
\begin{equation}\label{5.10}
T_1=\sum_{q=1}^\infty \sum^q_{\substack{a=1\\ (a,q)=1}}\left( q^{-1}S(q,a)\right)^u\Ome(n;q,a),
\end{equation}
where $\Ome(n;q,a)$ is equal to
$$\underset{m_1^k+\ldots +m_l^k\le n}{\sum_{|m_1|\le P}\ldots \sum_{|m_l|\le P}}
(n-m_1^k-\ldots -m_l^k)^{u/k-1}e(-(n-m_1^k-\ldots -m_l^k)a/q).$$
We sort the summands into arithmetic progressions modulo $q$ and recall (\ref{4.7}). Thus we see that
$$\Ome(n;q,a)=\sum_{r_1=1}^q\ldots \sum_{r_l=1}^q \Xi_{q,\bfr}^{(l)}(P;u/k-1)
e(-(n-r_1^k-\ldots -r_l^k)a/q).$$

\par When $1\le q\le P$, we apply Lemma \ref{lemma4.4} with $N=J+1$ to obtain
\begin{equation}\label{5.11}
\Ome(n;q,a)=2^l\frac{\Gam(u/k)\Gam(1+1/k)^l}{\Gam((u+l)/k)}n^{(u+l)/k-1}T_2+O(T_3),
\end{equation}
where
$$T_2=q^{-l}\sum_{r_1=1}^q\ldots \sum_{r_l=1}^qe(-(n-r_1^k-\ldots -r_l^k)a/q)$$
and
$$T_3=q^lP^{u-k}(q/P)^{J+1-l}\ll q^{J+1/(2k)}P^{u+l-k-J-1/(2k)}.$$
When $q>P$, meanwhile, one has the trivial estimate $\Ome(n;q,a)\ll P^{u+l-k}$. On recalling
 (\ref{1.7}), we find that $T_2=(q^{-1}S(q,a))^le(-na/q)$. Thus, on substituting (\ref{5.11}) into
 (\ref{5.10}) and recalling (\ref{5.4}) and (\ref{5.5}), we discern
that
\begin{equation}\label{5.12}
T_1=2^l\frac{\Gam(u/k)\Gam(1+1/k)^l}{\Gam((u+l)/k)}\grS_{u+l}(n;P)n^{(u+l)/k-1}+O(T_4),
\end{equation}
where
$$T_4=P^{u+l-k-J-1/(2k)}V_1^P(u;J+1/(2k))+P^{u+l-k}V_P^\infty (u;0).$$
In view of our hypothesis on $u$, an application of Lemma \ref{lemma5.1} yields the bound 
$T_4\ll P^{u+l-k-J-1/(2k)}$. Then on recalling the first conclusion of Lemma \ref{lemma5.2}, we deduce
 from (\ref{5.12}) that
$$T_1=2^l\frac{\Gam(u/k)\Gam(1+1/k)^l}{\Gam((u+l)/k)}\grS_{u+l}(n)n^{(u+l)/k-1}+
O(P^{u+l-k-J-1/(2k)}).$$
Making use of this estimate within (\ref{5.9}), therefore, we obtain the asymptotic formula claimed in
the statement of the lemma, and thus the proof of the lemma is complete.
\end{proof}

\section{Combining the major arc contributions, for even $k$}
Having evaluated asymptotically the expression $\grK_{u,l}(n)$, under appropriate conditions on $u$, we
 next seek to assemble the contributions comprising $\grI_{t,l}(n)$, and thereby evaluate $R_s(n)$.

\begin{lemma}\label{lemma6.1}
When $l$ and $t$ are natural numbers with $t-l\ge (J+1)k+2+\del_k$, one has
$\grI_{t,l}(n)=o(P^{t-k-J})$. Meanwhile, one has
$$\grI_{t,0}(n)=2^t\frac{\Gam(1+1/k)^t}{\Gam(t/k)}n^{t/k-1}\grS_t(n)+o(P^{t-k-J}).$$
\end{lemma}

\begin{proof} It follows from (\ref{3.6}), (\ref{5.1}) and the binomial theorem that
\begin{align*}
\grI_{t,l}(n)&=\sum_{v=0}^l (-1)^v\binom{l}{v}\int_\grM (2f^*(\alp))^{t-l+v}h(\alp)^{l-v}
e(-n\alp)\d\alp \\
&=\sum_{v=0}^l(-1)^v\binom{l}{v}\grK_{t-l+v,l-v}(n).
\end{align*}
Suppose temporarily that $l\ge 1$. Then we find from Lemma \ref{lemma5.3} that
$$\grI_{t,l}(n)=2^t\frac{\Gam(1+1/k)^t}{\Gam(t/k)}\grS_t(n)n^{t/k-1}\sum_{v=0}^l(-1)^v
\binom{l}{v}+o(P^{t-k-J}).$$
The first conclusion of the lemma consequently follows by noting that
$$\sum_{v=0}^l(-1)^v\binom{l}{v}=(1-1)^l=0\quad (l\ge 1).$$
When $l=0$, meanwhile, the desired conclusion follows directly from Lemma \ref{lemma5.2}. This
 completes the proof of the lemma.
\end{proof}

We are now equipped to prove Theorem \ref{theorem1.1}. Suppose that $s$ is $J$-admissible for $k$.
 Observe first that, as a consequence of Lemma \ref{lemma6.1}, one finds that whenever $0\le r\le J$ and
 $s-2J\ge (J+1)k+2+\del_k$, then
$$\sum_{l=0}^{2J-2r}\binom{s-r}{l}\grI_{s-r,l}(n)=2^{s-r}
\frac{\Gam(1+1/k)^{s-r}}{\Gam((s-r)/k)}\grS_{s-r}(n)n^{(s-r)/k-1}+o(P^{s-k-J}).$$
We therefore deduce from (\ref{3.9}) that whenever $s\ge (J+1)(k+2)+\del_k$, then
$$R_s(n)=\sum_{r=0}^J\binom{s}{r}\left(-\tfrac{1}{2}\right)^r
\frac{\Gam(1+1/k)^{s-r}}{\Gam((s-r)/k)}\grS_{s-r}(n)n^{(s-r)/k-1}+o(n^{(s-J)/k-1}).$$
On recalling (\ref{1.5}) and (\ref{1.6}), we find that the proof of Theorem \ref{theorem1.1} is complete.
\vskip.2cm

We remark that in the analysis yielding Lemmata \ref{lemma4.3} and \ref{lemma4.4}, it is the vanishing
 of high order derivatives that eliminates the potential existence of additional terms in the asymptotic
 formula delivered by Theorem \ref{theorem1.1}. In circumstances in which $\tet$ is an integer, one may
 apply the Euler-MacLaurin summation formula to obtain additional terms in Lemma \ref{lemma4.3} when
 $N>\lceil \tet\rceil$, and presumably these would lead to a zoo of additional terms of order 
$n^{(s-J)/k-1}$ in the asymptotic formula for $R_s(n)$ when $s$ is a multiple of $k$ and 
$J>\lceil s/k-1\rceil$.

\section{Preliminary man\oe uvres, for odd $k$} Our approach to proving Theorem \ref{theorem1.2} is
 broadly similar to that employed in the proof of Theorem \ref{theorem1.1}. Although we are
 consequently able to economise in our exposition, numerous technical complications force us to discuss
 this odd situation separately. We now suppose that $s$ and $k$ are natural numbers with $s>k\ge 3$ and
 $k$ odd. On this occasion we consider directly the number $R_s(n)$ of integral representations of $n$ in
 the shape (\ref{2.1}) with $1\le x_i\le P$ $(1\le i\le s)$. When $\grB$ is measurable, we put
\begin{equation}\label{7.1}
R_s(n;\grB)=\int_\grB f(\alp)^se(-n\alp)\d\alp .
\end{equation}
Thus, by orthogonality, one has
\begin{equation}\label{7.2}
R_s(n)=R_s(n;\grM)+R_s(n;\grm).
\end{equation}

\begin{lemma}\label{lemma7.1}
When $s$ is $J$-admissible for $k$, one has $R_s(n;\grm)=o(P^{s-k-J})$.
\end{lemma}

\begin{proof} On applying the triangle inequality, the desired conclusion is immediate from the definition
 of a $J$-admissible exponent.
\end{proof}

\section{The major arc contribution truncated, for odd $k$} Recalling the definition (\ref{3.2}) of
 $f^*(\alp)$, an application of the binomial theorem within (\ref{7.1}) reveals that
\begin{equation}\label{8.1}
R_s(n;\grM)=\sum_{l=0}^s\binom{s}{l}\grI_{s,l}^\dag (n),
\end{equation}
where
\begin{equation}\label{8.2}
\grI_{s,l}^\dag(n)=\int_\grM f^*(\alp)^{s-l}\left( f(\alp)-f^*(\alp)\right)^le(-n\alp)\d\alp .
\end{equation}

\begin{lemma}\label{lemma8.1} Suppose that $J$ is a non-negative integer. Then whenever $l>2J$ and
$s\ge \max\{l, k+2J+4\}$, one has $\grI_{s,l}^\dag(n)=o(P^{s-k-J})$.
\end{lemma}

\begin{proof} The argument of the proof of Lemma \ref{lemma3.1} applies, mutatis mutandis, to confirm
 the conclusion of the lemma by noting (\ref{3.3}).
\end{proof}

When $s\ge k+2J+4$, the conclusion of Lemma \ref{lemma8.1} combines with (\ref{7.2}), Lemma
 \ref{lemma7.1} and (\ref{8.1}) to deliver the formula
\begin{equation}\label{8.3}
R_s(n)=\sum_{l=0}^{2J}\binom{s}{l}\grI_{s,l}^\dag (n)+o(P^{s-k-J}).
\end{equation}

\section{An auxiliary lemma, for odd $k$} We now apply Lemmata \ref{lemma4.1} and \ref{lemma4.2} to
 obtain an asymptotic formula for an auxiliary sum of use for odd $k$. Let $X$ be a large positive real
 number, and let $\tet$ be a non-negative real number. When $q\in \dbN$ and $r\in \dbZ$, we define
\begin{equation}\label{9.1}
\Ups_{q,r}^\dag(X;\tet)=\sum_{-r/q<h\le (X-r)/q}(X^k-(qh+r)^k)^\tet .
\end{equation}

\begin{lemma}\label{lemma9.1} When $1\le N\le \lceil \tet\rceil$, one has
\begin{align*}
\Ups_{q,r}^\dag(X;\tet)=&\, q^{-1}X^{k\tet+1}\frac{\Gam(1+\tet)\Gam(1+1/k)}{\Gam(1+\tet+1/k)}
+\Psi+O\left(X^{k\tet}(q/X)^{N-1}\right) ,
\end{align*}
where
$$\Psi=X^{k\tet}\sum_{0\le \nu\le (N-1)/k}
\frac{\Gam(1+\tet)}{\nu!(\nu k+1)\Gam(1+\tet-\nu)}\bet_{\nu k+1}(-r/q)(q/X)^{k\nu}.$$
\end{lemma}

\begin{proof} We apply Lemma \ref{lemma4.2} with $F$ and $G$ given by (\ref{4.1a}). Write $a=-r/q$
 and $b=(X-r)/q$. Then one finds that $G(a)=X^k$ and $G(b)=0$. Moreover, since the formula (\ref{4.2})
 remains valid, one has $G^{(j)}(a)=0$ for $1\le j<k$, and also for $j>k$, and $G^{(k)}(a)=-k!q^k$.
 The formula (\ref{4.3}) also remains valid. From Lemma \ref{lemma4.2}, we thus deduce that $F(G(x))$
 has continuous derivatives through the $N$th order on $(a,b)$, continuous derivatives through the
 $(N-1)$-st order on $[a,b]$, and further $|{\rm d}^N F(G(x))/{\rm d}x^N|$ is integrable on $[a,b]$. 
Note also that when $0\le m<\tet$, one has $F^{(m)}(G(b))=0$, and hence
$$\left.\frac{{\rm d}^\kap\ }{{\rm d}x^\kap}F(G(x))\right|_{x=b}=0\quad (0\le \kap<\tet).$$
In addition, it follows from Lemma \ref{lemma4.2} that
$$\left.\frac{{\rm d}^\kap\ }{{\rm d}x^\kap}F(G(x))\right|_{x=a}=0\quad (0\le \kap<\tet),$$
except possibly when $\kap$ is divisible by $k$, say $\kap=\nu k$, in which case
\begin{align*}
\left.\frac{{\rm d}^{\nu k}\ }{{\rm d}x^{\nu k}}F(G(x))\right|_{x=a}&=
\frac{(\nu k)!}{\nu!}F^{(\nu)}(G(a))\left( \frac{G^{(k)}(a)}{k!}\right)^\nu \\
&=\frac{(\nu k)!}{\nu!}\tet (\tet -1)\ldots (\tet-\nu +1)X^{k\tet-k\nu}(-1)^\nu q^{k\nu}.
\end{align*}

\par On substituting these values into Lemma \ref{lemma4.1}, we see that
\begin{align}
\sum_{a<h\le b}F(G(h))=\int_a^bF(G(x))&\d x-\sum_{0\le \nu \le (N-1)/k}T_\nu \notag \\
&\, -\frac{(-1)^N}{N!}\int_a^b\bet_N(x)\frac{{\rm d}^N\ }{{\rm d}x^N}F(G(x))\d x,\label{9.2}
\end{align}
where
$$T_\nu=\frac{(-1)^{\nu k+1+\nu}}{(\nu k+1)!}\frac{(\nu k)!}{\nu!}\frac{\Gam(1+\tet)}
{\Gam(1+\tet-\nu)}X^{k(\tet-\nu)}q^{k\nu}\bet_{\nu k+1}(a).$$
By making the change of variable $y=(qx+r)/X$, we find that
\begin{align}
\int_a^bF(G(x))\d x&=q^{-1}X^{k\tet+1}\int_0^1(1-y^k)^\tet\d y\notag \\
&=q^{-1}X^{k\tet+1}\frac{\Gam(1+\tet)\Gam(1+1/k)}{\Gam(1+\tet+1/k)}.\label{9.3}
\end{align}
Also, just as in the corresponding treatment described in the argument of the proof of Lemma
 \ref{lemma4.3}, one finds that
$$\int_a^b\bet_N(x)\frac{{\rm d}^N\ }{{\rm d}x^N}F(G(x))\d x\ll X^{k\tet}(q/X)^{N-1}.$$
On recalling that $k$ is odd and $a=-r/q$, the conclusion of the lemma follows on substituting this
 estimate together with (\ref{9.3}) into (\ref{9.2}), and then recalling the definition (\ref{9.1}) of
 $\Ups_{q,r}^\dag(X;\tet)$.
\end{proof}

We extend the previous conclusion so as to handle a multidimensional generalisation. When $q\in \dbN$
 and $r_1,\ldots ,r_l\in \dbZ$, we define
\begin{equation}\label{9.4}
\Xi_{q,\bfr}^{\dag(l)}(X;\tet)=\sum_{\substack{0<x_1\le X\\ x_1\equiv r_1\mmod{q}}}\ldots
\sum_{\substack{0<x_l\le X\\ x_l\equiv r_l\mmod{q}}}(X^k-x_1^k-\ldots -x_l^k)^\tet ,
\end{equation}
where the summands are constrained by the inequality $x_1^k+\ldots +x_l^k\le X^k$. It is convenient
 also to introduce a multidimensional analogue of the Bernoulli polynomials specific to the purpose at
 hand. Let $\sig_m(y_1,\ldots ,y_l)$ denote the $m$th elementary symmetric polynomial in 
$y_1,\ldots ,y_l$, and define
$$B_m^{(l)}(q;\bfr)=\sig_m(\bet_1(-r_1/q),\ldots ,\bet_1(-r_l/q)).$$
Note that $\sig_0(y_1,\ldots ,y_l)=1$, and by convention $\sig_{-1}(y_1,\ldots ,y_l)=0$.

\begin{lemma}\label{lemma9.2} When $1\le N\le \min \{ \lceil \tet\rceil,k+1\}$, one has
\begin{align*}
\Xi_{q,\bfr}^{\dag(l)}(X;\tet)=&\,X^{k\tet}\sum_{m=0}^l
\frac{\Gam(1+\tet)\Gam(1+1/k)^{l-m}}{\Gam(1+\tet+(l-m)/k)}B_m^{(l)}(q;\bfr)(X/q)^{l-m}\\
&\, +O(X^{k\tet}(q/X)^{N-1}(1+X/q)^{l-1}).
\end{align*}
\end{lemma}

\begin{proof} We proceed by induction on $l$, noting that the case $l=1$ is already a consequence of
 Lemma \ref{lemma9.1}. Suppose then that $L>1$, and that the desired conclusion has been established
 for $1\le l<L$. From (\ref{9.4}), we obtain
\begin{equation}\label{9.5}
\Xi_{q,\bfr}^{\dag(l)}(X;\tet)=\sum_{-r_L/q<h_L\le (X-r_L)/q}\Xi_{q,\bfr'}^{\dag(L-1)}(Y;\tet),
\end{equation}
where
$$\bfr'=(r_1,\ldots ,r_{L-1}) \quad \text{and}\quad Y=(X^k-(qh_L+r_L)^k)^{1/k}.$$
Our inductive hypothesis delivers the asymptotic formula
\begin{align*}
\Xi_{q,\bfr'}^{\dag(L-1)}(Y;\tet)=&\,Y^{k\tet}\sum_{m=0}^{L-1}
\frac{\Gam(1+\tet)\Gam(1+1/k)^{L-m-1}}{\Gam(1+\tet+(L-m-1)/k)}B_m^{(L-1)}(q;\bfr')
(Y/q)^{L-m-1}\\
&\, +O(Y^{k\tet}(q/Y)^{N-1}(1+Y/q)^{L-2}).
\end{align*}
By substituting this expression into (\ref{9.5}), we deduce that
\begin{align}
\Xi_{q,\bfr}^{\dag(L)}(X;\tet)=\sum_{m=0}^{L-1}&
\frac{\Gam(1+\tet)\Gam(1+1/k)^{L-m-1}}{\Gam(1+\tet+(L-m-1)/k)}B_m^{(L-1)}(q;\bfr')
q^{m+1-L}T_m\notag \\
&\, +O(X^{k\tet}(q/X)^{N-1}(1+X/q)^{L-1}),\label{9.6}
\end{align}
where
\begin{equation}\label{9.7}
T_m=\sum_{-r_L/q<h_L\le (X-r_L)/q}(X^k-(qh_L+r_L)^k)^{\tet+(L-m-1)/k}.
\end{equation}

\par An application of Lemma \ref{lemma9.1} leads from (\ref{9.7}) to the asymptotic formula
\begin{align*}
T_m=&\, q^{-1}X^{k\tet+L-m}\frac{\Gam(1+\tet+(L-m-1)/k)\Gam(1+1/k)}{\Gam(1+\tet+(L-m)/k)}\\
&\, +X^{k\tet+L-m-1}\bet_1(-r_L/q)+O(X^{k\tet+L-m-1}(q/X)^{N-1}),
\end{align*}
whence from (\ref{9.6}) one obtains the relation
\begin{align*}
\Xi_{q,\bfr}^{\dag(L)}(X;\tet)=X^{k\tet}\sum_{m=0}^L&\,
\frac{\Gam(1+\tet)\Gam(1+1/k)^{L-m}}{\Gam(1+\tet+(L-m)/k)}C_m(q,\bfr)(X/q)^{L-m}\\
&\, +O\left( X^{k\tet}(q/X)^{N-1}(1+X/q)^{L-1}\right) ,
\end{align*}
where
$$C_m(q,\bfr)=B_m^{(L-1)}(q;\bfr')+\bet_1(-r_L/q)B_{m-1}^{(L-1)}(q;\bfr').$$
By considering the relevant symmetric polynomials, one sees that $C_m(q,\bfr)=B_m^{(L)}(q;\bfr)$.
 Thus we conclude that the inductive hypothesis holds for $l=L$, confirming the inductive step and
 completing the proof of the lemma.
\end{proof}

In Lemma \ref{lemma9.2} we have limited the parameter $N$ to be at most $k+1$ in order that terms
 involving $\bet_{\nu k+1}(-r_i/q)$ with $\nu\ge 1$ be absent. A more detailed investigation reveals that
 such additional terms can be accommodated at the expense of substantial complications.

\section{The major arc contribution evaluated, for odd $k$} We turn next to the evaluation of the integral
$\grI_{s,l}^\dag(n)$ defined in (\ref{8.2}). With this objective in mind, we consider the auxiliary integral
\begin{equation}\label{10.1}
\grK_{u,l}^\dag (n)=\int_\grM f^*(\alp)^uf(\alp)^le(-n\alp)\d\alp .
\end{equation}
Making use of the definition (\ref{1.3}) of $f(\alp)$, and recalling (\ref{5.3}), one finds that
\begin{equation}\label{10.2}
\grK_{u,l}^\dag(n)=\sum_{1\le m_1\le P}\ldots \sum_{1\le m_l\le P}\grR_u(n-m_1^k-\ldots -m_l^k).
\end{equation}
Recall the exponential sum $T(q,a)$ defined in (\ref{1.7}), and the modified singular series 
$\grS_{s,j}(n)$ defined in (\ref{1.8}). It is useful also to define the truncation
\begin{equation}\label{10.3}
\grS_{s,j}(n;Q)=\sum_{1\le q\le Q}\sum^q_{\substack{a=1\\ (a,q)=1}}
\left( q^{-1}S(q,a)\right)^{s-j}T(q,a)^je(-na/q).
\end{equation}
These modified singular series have good convergence properties, as a consequence of the following
 simple estimate for $T(q,a)$.

\begin{lemma}\label{lemma10.1} Suppose that $a\in \dbZ$ and $q\in \dbZ$ satisfy $(q,a)=1$. Then for
 each $\eps>0$, one has $T(q,a)\ll q^{1/2+\eps}$.
\end{lemma}

\begin{proof} We find from (\ref{1.7}) that
\begin{equation}\label{10.4}
T(q,a)=-\tfrac{1}{2}S(q,a)+q^{-1}\sum_{r=1}^q\int_r^q e(ar^k/q)\d x.
\end{equation}
By interchanging the order of summation and integration here, we deduce from 
\cite[equation (4.14)]{Vau1997} that
\begin{align*}
T(q,a)+\tfrac{1}{2}S(q,a)&=q^{-1}\int_0^q \sum_{1\le r\le x} e(ar^k/q)\d x\\
&=q^{-1}\int_0^q\left( q^{-1}S(q,a)x+O(q^{1/2+\eps}) \right)\d x\\
&=\tfrac{1}{2}S(q,a)+O(q^{1/2+\eps}),
\end{align*}
and the lemma follows at once.
\end{proof}

\begin{lemma}\label{lemma10.2}
When $t>k+r+1$, one has
$$\grS_{t,r}(n;Q)\ll 1+(Q^{1/k})^{r(1+k/2)+k-(t-3/2)}.$$
Moreover, the modified singular series $\grS_{t,r}(n)$ is absolutely convergent whenever 
$t\ge \tfrac{1}{2}(r+2)(k+2)$.
\end{lemma}

\begin{proof} On recalling (\ref{5.4}) and (\ref{10.3}), one finds from Lemmata \ref{lemma5.1} and
 \ref{lemma10.1} that
$$\grS_{t,r}(n;Q)\ll V_1^{Q+1}(t-r;r/2+\eps)\ll 1+(Q^{1/k})^{(r+2)(k+2)/2-t-1/2}.$$
This confirms the first assertion of the lemma. The second follows on observing that the hypothesis
$t\ge \tfrac{1}{2}(r+2)(k+2)$ ensures in like manner that
$$\sum_{1\le q\le Q}\sum^q_{\substack{a=1\\ (a,q)=1}}|q^{-1}S(q,a)|^{t-r}|T(q,a)|^r\ll 
V_1^{Q+1}(t-r;r/2+\eps )\ll 1.$$
\end{proof}

We are now equipped to evaluate $\grK_{u,l}^\dag(n)$.

\begin{lemma}\label{lemma10.3}
Suppose that $J$ and $u$ are non-negative integers with $J\le k$ and $u\ge (J+1)k+2$. Then there is a
 positive number $\eta$ for which
$$\grK_{u,l}^\dag(n)=n^{(u+l)/k-1}\sum_{m=0}^{\min\{l,J\}}\grD_m n^{-m/k}+
O(P^{u+l-k-J-\eta}),$$
where
$$\grD_m=\binom{l}{m}\frac{\Gam(1+1/k)^{u+l-m}}{\Gam((u+l-m)/k)}\grS_{u+l,m}(n).$$
\end{lemma}

\begin{proof} On recalling the formula (\ref{10.2}) for $\grK_{u,l}^\dag(n)$, we find from Lemma
 \ref{lemma5.2} that there is a positive number $\eta$ such that
\begin{equation}\label{10.5}
\grK_{u,l}^\dag (n)=\frac{\Gam(1+1/k)^u}{\Gam(u/k)}T_1+O(P^{u+l-k-J-\eta}),
\end{equation}
where
$$T_1=\underset{m_1^k+\ldots +m_l^k\le n}{\sum_{1\le m_1\le P}\ldots \sum_{1\le m_l\le P}}
\grS_u(n-m_1^k-\ldots -m_l^k)(n-m_1^k-\ldots -m_l^k)^{u/k-1}.$$
Applying the definition (\ref{1.2}) of the singular series, we find that
\begin{equation}\label{10.6}
T_1=\sum_{q=1}^\infty \sum^q_{\substack{a=1\\ (a,q)=1}}(q^{-1}S(q,a))^u\Ome^\dag (n;q,a),
\end{equation}
where $\Ome^\dag(n;q,a)$ is equal to
$$\underset{m_1^k+\ldots +m_l^k\le n}{\sum_{1\le m_1\le P}\ldots
\sum_{1\le m_l\le P}}(n-m_1^k-\ldots -m_l^k)^{u/k-1}e(-(n-m_1^k-\ldots -m_l^k)/q).$$
By sorting summands into arithmetic progressions modulo $q$ and recalling (\ref{9.4}), we see that
$$\Ome^\dag(n;q,a)=\sum_{r_1=1}^q\ldots \sum_{r_l=1}^q\Xi_{q,\bfr}^{\dag(l)}(P;u/k-1)
e(-(n-r_1^k-\ldots -r_l^k)a/q).$$

\par When $1\le q\le P$, we apply Lemma \ref{lemma9.2} with $N=J+1$, obtaining
\begin{equation}\label{10.7}
\Ome^\dag(n;q,a)=n^{u/k-1}\sum_{m=0}^l\frac{\Gam(u/k)\Gam (1+1/k)^{l-m}}{\Gam((u+l-m)/k)}
n^{(l-m)/k}U_m+O(U^*),
\end{equation}
where
$$U_m=q^{m-l}\sum_{r_1=1}^q\ldots \sum_{r_l=1}^q B_m^{(l)}(q;\bfr)
e(-(n-r_1^k-\ldots -r_l^k)a/q)$$
and
$$U^*=q^lP^{u-k}(q/P)^{J+1-l}\ll q^{J+1/(2k)}P^{u+l-k-J-1/(2k)}.$$
When $q>P$, meanwhile, one has the trivial estimate $\Ome^\dag(n;q,a)\ll P^{u+l-k}$. On recalling
 (\ref{1.7}), we find that
$$U_m=\binom{l}{m}(q^{-1}S(q,a))^{l-m}T(q,a)^me(-na/q).$$
Thus, on substituting (\ref{10.7}) into (\ref{10.6}) and recalling (\ref{5.4}) and (\ref{10.3}), we discern
 that
\begin{equation}\label{10.8}
T_1=n^{u/k-1}\sum_{m=0}^l\binom{l}{m}\frac{\Gam(u/k)\Gam(1+1/k)^{l-m}}{\Gam((u+l-m)/k)}
\grS_{u+l,m}(n;P)n^{(l-m)/k}+O(T_2),
\end{equation}
where
$$T_2=P^{u+l-k-J-1/(2k)}V_1^P(u;J+1/(2k))+P^{u+l-k}V_P^\infty(u;0).$$

\par In view of our hypothesis on $u$, an application of Lemma \ref{lemma5.1} delivers the bound
$T_2\ll P^{u+l-k-J-1/(2k)}$. In addition, by applying Lemma \ref{lemma5.1} via (\ref{5.4}) to
(\ref{1.8}) and (\ref{10.3}), one discerns that our hypothesis on $u$ ensures that when $0\le m\le J$,
 one has
$$\grS_{u+l,m}(n)-\grS_{u+l,m}(n;P)\ll V_P^\infty(u+l-m;m/2+\eps)\ll P^{m-J-1/(2k)}.$$
Meanwhile, when $m>J$, it follows from Lemma \ref{lemma10.2} that
$$\grS_{u+l,m}(n;P)\ll 1+P^{m+1-(u+l-3/2)/k}\ll P^{m-J-1/(2k)}.$$
On substituting these estimates into (\ref{10.8}), we conclude that
\begin{align*}
T_1-n^{(u+l)/k-1}&\,\sum_{m=0}^{\min\{l,J\}}\binom{l}{m}
\frac{\Gam(u/k)\Gam(1+1/k)^{l-m}}{\Gam((u+l-m)/k)}
n^{-m/k}\grS_{u+l,m}(n)\\
\ll &\, P^{u-k}\sum_{m=0}^l P^{l-m}(P^{m-J-1/(2k)})\ll P^{u+l-k-J-1/(2k)}.
\end{align*}
The conclusion of the lemma now follows from (\ref{10.5}).
\end{proof}

\section{Combining the major arc contributions, for odd $k$} We now reassemble the integrals
 $\grI_{s,l}^\dag(n)$ so as to evaluate $R_s(n)$.

\begin{lemma}\label{lemma11.1}
Suppose that $0\le J\le k$, and that $l$ and $s$ are natural numbers with $s-l\ge (J+1)k+2$. Then
 whenever $l>J$, one has $\grI_{s,l}^\dag(n)=o(P^{s-k-J})$. Meanwhile, when instead $l\le J$, one has
$$\grI_{s,l}^\dag(n)=\frac{\Gam(1+1/k)^{s-l}}{\Gam((s-l)/k)}\grS_{s,l}(n)n^{(s-l)/k-1}
+o(P^{s-k-J}).$$
\end{lemma}

\begin{proof} It follows from (\ref{8.2}), (\ref{10.1}) and the binomial theorem that
\begin{align*}
\grI_{s,l}^\dag(n)&=\sum_{v=0}^l(-1)^v\binom{l}{v}\int_\grM f^*(\alp)^{s-l+v}f(\alp)^{l-v}
e(-n\alp)\d\alp \\
&=\sum_{v=0}^l(-1)^v\binom{l}{v}\grK_{s-l+v,l-v}^\dag(n).
\end{align*}
Then we find from Lemma \ref{lemma10.3} that
$$\grI_{s,l}^\dag(n)=\sum_{v=0}^l(-1)^v\binom{l}{v}\sum_{m=0}^{\min\{l-v,J\}}
\binom{l-v}{m}\grB_m+o(P^{s-k-J}),$$
where
$$\grB_m=\frac{\Gam(1+1/k)^{s-m}}{\Gam((s-m)/k)}\grS_{s,m}(n)n^{(s-m)/k-1}.$$
On making use of the identity
$$\binom{l}{v}\binom{l-v}{m}=\binom{l}{m}\binom{l-m}{v},$$
therefore, we deduce that
\begin{equation}\label{11.1}
\grI_{s,l}^\dag(n)=\sum_{m=0}^{\min\{ l,J\}}\binom{l}{m}\sum_{v=0}^{l-m}(-1)^v
\binom{l-m}{v}\grB_m+o(P^{s-k-J}).
\end{equation}

\par When $m<l$, one has
$$\sum_{v=0}^{l-m}(-1)^v\binom{l-m}{v}=(1-1)^{l-m}=0,$$
so that only the terms with $m=l$ contribute in (\ref{11.1}). Thus we see that when $l>J$, the outermost
 sum on the right hand side of (\ref{11.1}) contributes nothing, and this confirms the first conclusion of
the lemma. When $l\le J$, meanwhile, the only contribution comes from those terms with $m=l$ and
 $v=0$, and in this way one obtains the second conclusion of the lemma.
\end{proof}

We are now equipped to prove Theorem \ref{theorem1.2}. Suppose that $0\le J\le k$, and that $s$ is
$J$-admissible for $k$. Observe first that, as a consequence of Lemma \ref{lemma11.1}, one finds that
 whenever $s-2J\ge (J+1)k+2$, then
$$\sum_{l=0}^{2J}\binom{s}{l}\grI_{s,l}^\dag (n)=\sum_{l=0}^J\binom{s}{l}
\frac{\Gam(1+1/k)^{s-l}}{\Gam((s-l)/k)}\grS_{s,l}(n)n^{(s-l)/k-1}+o(P^{s-k-J}).$$
We therefore deduce from (\ref{8.3}) that whenever $s\ge (J+1)(k+2)$, then
\begin{equation}\label{11.2}
R_s(n)=\sum_{l=0}^J\binom{s}{l}\frac{\Gam(1+1/k)^{s-l}}{\Gam((s-l)/k)}\grS_{s,l}(n)
n^{(s-l)/k-1}+o(n^{(s-J)/k-1}).
\end{equation}
On recalling (\ref{1.5}) and (\ref{1.9}), we find that the proof of Theorem \ref{theorem1.2} is complete.
\vskip.2cm

We have yet to discuss the modified singular series $\grS_{s,l}(n)$.  It is clear, however, that the limitation
$J\le k$ can be removed if one is prepared to endure further analysis in which exponential sums of the
 shape
$$\sum_{r=1}^q\bet_{k+1}(-r/q)e(ar^k/q),$$
and yet more exotic creatures, appear. The complexity rises rapidly, and we avoid discussion of such
matters in the absence of deserving applications.\par

Notice also that when $k$ is even, one has
$$T(q,a)=\sum_{r=1}^q\left( \frac{1}{2}-\frac{r}{q}\right)e(ar^k/q)=\sum_{r=0}^{q-1}\left(
\frac{1}{2}-\frac{q-r}{q}\right) e(ar^k/q),$$
so that $T(q,a)=-1-T(q,a)$. We therefore see that when $k$ is even, one has $T(q,a)=-\frac{1}{2}$,
and hence it follows from (\ref{1.8}) that $\grS_{s,j}(n)=\left( -\frac{1}{2}\right)^j\grS_{s-j}(n)$. The
 asymptotic formula (\ref{11.2}) is therefore consistent with that delivered by Theorem \ref{theorem1.1},
 at least in those restricted circumstances where $J\le k$.

\section{Exceptional sets} Our goal in this section is to establish Theorem \ref{theorem1.3}. We take an
 abbreviated approach, concentrating on the contribution of the minor arcs. Let $N$ be a large positive
 number, and put $P=N^{1/k}$. We assume that the exponent $2s$ is $2J$-admissible for $k$. Thus, for
 some positive function $L(t)$ growing sufficiently slowly in terms of $t$, and with
 $L(t)\rightarrow +\infty$ as $t\rightarrow \infty$, one has
$$\int_\grm |f(\alp)|^{2s}\d\alp \ll P^{2s-k-2J}L(P)^{-3}.$$
Define the function $F(\alp)$ by taking $F(\alp)=f(\alp)^s$ when $\alp\in\grm$, and otherwise by taking
 $F(\alp)=0$. Also, let $\hat F(n)$ be the Fourier coefficient of $F$, so that
$$\hat F(n)=\int_0^1 F(\alp)e(-\alp n)\d\alp = \int_\grm f(\alp)^se(-n\alp)\d\alp.$$
Then by Bessel's inequality, one has
$$\sum_{n\in \dbZ}|\hat F(n)|^2 \le \int_\grm |f(\alp)|^{2s}\d\alp \ll P^{2s-k-2J}L(P)^{-3}.$$
Let $\calZ_s(N)$ denote the set of integers $n$ with $N/2<n\le N$ for which one has
$$|\hat F(n)|>P^{s-k-J}L(P)^{-1}.$$
We write $Z$ for $\text{card}(\calZ_s(N))$. Then it is immediate that $Z\ll P^kL(P)^{-1}$, and thus we
 see that for almost all integers $n$ with $N/2<n\le N$, one has
$$\int_\grm f(\alp)^se(-n\alp)\d\alp =o(P^{s-k-J}).$$
Of course, with only modest adjustments in this argument, one may show also that for almost all integers
 $n$ with $N/2<n\le N$, one has likewise
$$\int_\grm h(\alp)^se(-n\alp)\d\alp =o(P^{s-k-J}).$$

\par The corresponding major arc contributions
$$\int_\grM f(\alp)^se(-n\alp)\d\alp \quad \text{and}\quad \int_\grM h(\alp)^se(-n\alp)\d\alp $$
are obtained by means of the work of \S\S2--11, with inconsequential modification. Thus the conclusions
 of Theorem \ref{theorem1.3} follow from the work that we have already completed, after summing over
 dyadic intervals.

\section{The modified singular series $\grS_{s,r}(n)$} It remains to discuss the modified singular series
 $\grS_{s,r}(n)$. These series are presumably non-zero in general. Such is the case when $k$ is even, for
then as we have noted one has
$$\grS_{s,r}(n)=\left( -\tfrac{1}{2}\right)^r\grS_{s-r}(n),$$
and under modest local conditions, the conventional singular series $\grS_{s-r}(n)$ is indeed non-zero.
When $k$ is odd, however, the situation is less clear.\par

We spend some time now examining the situation in which $k$ is odd, beginning with the proof of
 Theorem \ref{theorem1.4}. Write
$$T^\dag(q,a)=\sum_{r=0}^q\left( \frac{1}{2}-\frac{r}{q}\right) e(ar^k/q).$$
We begin by noting that
$$T(q,a)=-\tfrac{1}{2}+T^\dag (q,a),$$
whence
$${\overline{T^\dag (q,a)}}=T^\dag (q,-a)=\sum_{r=0}^q\left( \frac{1}{2}-\frac{r}{q}\right)
e(a(q-r)^k/q)=-T^\dag (q,a).$$
It follows that $T^\dag (q,a)$ is purely imaginary. Observe that when $s\ge \tfrac{3}{2}(k+2)$, so that
 both $\grS_{s,1}(n)$ and $\grS_{s-1}(n)$ are absolutely convergent, one has
\begin{align*}
\grS_{s,1}(n)+\tfrac{1}{2}\grS_{s-1}(n)&=\sum_{q=1}^\infty \sum^q_{\substack{a=1\\ (a,q)=1}}
(q^{-1}S(q,a))^{s-1}\left(\tfrac{1}{2}+T(q,a)\right)e(-na/q)\\
&=\sum_{q=1}^\infty \sum^q_{\substack{a=1\\ (a,q)=1}}(q^{-1}S(q,a))^{s-1}
T^\dag (q,a)e(-na/q).
\end{align*}
In present circumstances, where $k$ is odd, one has $S(q,a)=S(q,-a)$, and thus we are led from our
earlier discussion to the interim conclusion
\begin{align*}
\grS_{s,1}(n)+\tfrac{1}{2}\grS_{s-1}(n)&=\sum_{q=1}^\infty \sum^q_{\substack{a=1\\ (a,q)=1}}
(q^{-1}S(q,-a))^{s-1}T^\dag (q,-a)e(na/q)\\
&=-\sum_{q=1}^\infty \sum^q_{\substack{a=1\\ (a,q)=1}}(q^{-1}S(q,a))^{s-1}T^\dag (q,a)e(na/q)\\
&=-\grS_{s,1}(-n)-\tfrac{1}{2}\grS_{s-1}(-n).
\end{align*}
The relation $S(q,a)=S(q,-a)$ similarly ensures that $\grS_{s-1}(n)=\grS_{s-1}(-n)$, and hence we
conclude that
\begin{equation}\label{13.1}
\grS_{s,1}(n)+\grS_{s,1}(-n)=-\grS_{s-1}(n).
\end{equation}

\par Observe next that when $s\ge \tfrac{3}{2}k+3$, then it follows from Lemma \ref{lemma10.1} via
(\ref{5.4}) and Lemma \ref{lemma5.1} that
$$\grS_{s,1}(n)-\grS_{s,1}(n;Q)\le V_Q^\infty(s-1;\tfrac{1}{2}+\eps)\ll Q^{-1/(2k)}.$$
Note also that when $Q$ is a natural number and $n$ is a multiple of $Q!$, then for $1\le q\le Q$ one has
$e(-na/q)=1=e(na/q)$. Thus, with the same assumptions on $n$, one has
\begin{align*}
\grS_{s,1}(n)&=\sum_{1\le q\le Q}\sum^q_{\substack{a=1\\ (a,q)=1}}
(q^{-1}S(q,a))^{s-1}T(q,a)e(-na/q)+O(Q^{-1/(2k)})\\
&=\sum_{1\le q\le Q}\sum^q_{\substack{a=1\\ (a,q)=1}}
(q^{-1}S(q,a))^{s-1}T(q,a)e(na/q)+O(Q^{-1/(2k)})\\
&=\grS_{s,1}(-n)+O(Q^{-1/(2k)}).
\end{align*}
Thus we deduce from (\ref{13.1}) that
$$\grS_{s,1}(n)=-\tfrac{1}{2}\grS_{s-1}(n)+O(Q^{-1/(2k)}).$$
This confirms the conclusion of Theorem \ref{theorem1.4}.\par

When $r\ge 2$, the behaviour of $\grS_{s,r}(n)$ is less clear, since $T^\dag (q,a)^r$ is real whenever $r$
 is even, and the above device fails. However, some information concerning non-vanishing of linear
combinations of the series $\grS_{s,r}(n)$ would be available with additional work.\par

We finish by proving Theorem \ref{theorem1.5}. Suppose that $s\ge \tfrac{1}{2}(r+2)(k+2)$. Then the
 conclusion of Lemma \ref{lemma10.2} shows that the modified singular series $\grS_s(n;r)$ is absolutely
 convergent. Thus there is a constant $c_r$ for which $|\grS_{s,r}(n)|\le c_r$. For the sake of concision,
 write $U(q,a)=S(q,a)^{s-r}T(q,a)^r$. Let $\eta$ be a positive number with $\eta<c_r$, and choose $Q$
 in such a way that
$$\sum_{q\ge Q}\sum^q_{\substack{a=1\\ (a,q)=1}}q^{r-s}|U(q,a)|<\eta .$$
Then from (\ref{1.8}) one discerns that
$$\sum_{1\le n\le x}|\grS_{s,r}(n)|^2\ge \sum_{1\le n\le x}
\Bigl| \sum_{1\le q<Q}\sum^q_{\substack{a=1\\ (a,q)=1}}q^{r-s}U(q,a)e(-na/q)\Bigr|^2-3c_r\eta x.$$
On squaring out the sums over $q$ and $a$, the diagonal contribution is
$$T_1=\lfloor x\rfloor \sum_{1\le q<Q}\sum^q_{\substack{a=1\\ (a,q)=1}}q^{2r-2s}|U(q,a)|^2,$$
whilst the off-diagonal terms make a contribution
$$T_2\ll \sum_{1\le q<Q}\sum^q_{\substack{a=1\\ (a,q)=1}}q^{r+1-s}|U(q,a)|
\sum_{1\le w<Q}\sum^w_{\substack{b=1\\ (b,w)=1\\ b/w\ne a/q}}w^{r+1-s}|U(w,b)|.$$

\par On recalling (\ref{5.4}), we find from Lemma \ref{lemma10.1} that
$$T_2\ll \Bigl( \sum_{1\le q<Q}\sum^q_{\substack{a=1\\ (a,q)=1}}q^{r+1-s}|U(q,a)|\Bigr)^2
\ll V_1^Q(s-r;\tfrac{1}{2}r+1+\eps)^2.$$
Hence, provided that $s\ge \frac{1}{2}(r+4)(k+2)$, we may apply Lemma \ref{lemma5.1} to deduce that
$T_2\ll 1$. Meanwhile, in a similar fashion, one sees that under the same conditions on $s$, one has
$$\sum_{1\le q<Q}\sum^q_{\substack{a=1\\ (a,q)=1}}q^{2r-2s}|U(q,a)|^2\ll V_1^Q(2s-2r;r+\eps)
\ll 1.$$
Hence the sum on the left hand side here converges as $Q\rightarrow \infty$. Since this series is clearly
 positive, it follows that for some $\del>0$ one has $T_1\ge 4\del^2x$. We now fix $\eta$ with
 $0<\eta<\del^2/(3c_r)$, and conclude that
$$\sum_{1\le n\le x}|\grS_{s,r}(n)|^2\ge 4\del^2x-3c_r\eta x+O(1)>2\del^2x.$$

\par The last sum is the key to our proof of Theorem \ref{theorem1.5}. The number of natural numbers
 $n$ with $1\le n\le x$ for which $|\grS_{s,r}(n)|\ge \del$ is at least
\begin{align*}
c_r^{-2}\sum_{\substack{1\le n\le x\\ |\grS_{s,r}(n)|\ge \del}}|\grS_{s,r}(n)|^2&\ge c_r^{-2}
\Bigl( \sum_{1\le n\le x}|\grS_{s,r}(n)|^2-\del^2x\Bigr)\\
&\ge c_r^{-2}(2\del^2x-\del^2x)>(\del/c_r)^2x.
\end{align*}
Thus we conclude that $|\grS_s(n;r)|\ge \del$ for a positive proportion of the integers $n$ with
$1\le n\le x$, and this completes the proof of Theorem \ref{theorem1.5}.

\bibliographystyle{amsbracket}
\providecommand{\bysame}{\leavevmode\hbox to3em{\hrulefill}\thinspace}

\end{document}